\documentclass[12pt]{amsart}

\usepackage{graphicx}
\usepackage{color}
\usepackage{url}
\usepackage{dsfont}
\usepackage{extarrows}
\usepackage{amssymb}

\newtheorem{theorem}{Theorem}[section]
\newtheorem{lemma}[theorem]{Lemma}

\newtheorem{cor}[theorem]{Corollary}

\theoremstyle{definition}
\newtheorem{defi}[theorem]{Definition}
\newtheorem{example}[theorem]{Example}

\theoremstyle{remark}
\newtheorem{remark}[theorem]{Remark}

\numberwithin{equation}{section}

\newcommand{\rr}{{\mathbb R}}

\newcommand{\Raum}{\mathbb{R}_{+} \times \mathbb{R}}
\newcommand{\raum}{\mathbb{R}_{+}}

\newcommand{\Rplus}{\mathbb{R}_{+}}

\newcommand{\ind}[1]{\mathds{1}_{#1}}

\newcommand{\Konvergenzcinfty}{\xrightarrow[c \rightarrow \infty]{}}
\newcommand{\wKonvergenz}{\xlongrightarrow{w}}

\newcommand{\WKonvergenz}{\xLongrightarrow[n \rightarrow \infty]{}}

\newcommand{\wKonvergenzcinfty}{\xrightarrow[c \rightarrow \infty]{w}}

\newcommand{\JKonvergenzc}{\xlongrightarrow[c \rightarrow \infty]{J_1}}

\newcommand{\Dpu}{\mathbb{D}_{\uparrow,u}}

\newcommand{\intnu}{\int_{0}^{\infty}}
\newcommand{\sumnu}{\sum_{n=0}^{\infty}}
\newcommand{\indK}[1]{\mathds{1}_{\left\{#1\right\}}}

\begin{document}
\title{Coupled Continuous Time Random Maxima}
\author{Katharina Hees and Hans-Peter Scheffler} 
\address{Katharina Hees, Institut für medizinische Biometrie und Informatik, Universit\"at Heidelberg, 69120 Heidelberg, Germany} 
\email{hees@imbi.uni-heidelberg.de}
\address{Hans-Peter Scheffler, Deparment Mathematik, Universit\"at Siegen, 57072 Siegen, Germany} 
\email{Scheffler@mathematik.uni-siegen.de}

\begin{abstract}
Continuous Time Random Maxima (CTRM) are a generalization of classical extreme value theory: Instead of observing random events at regular intervals in time, the waiting times between the events are also random variables with arbitrary distributions. In case that the waiting times between the events have infinite mean, the limit process that appears differs from the limit process that appears in the classical case. With a continuous mapping approach we derive a limit theorem for the case that the waiting times and the subsequent events are dependent and for the case that the waiting times dependent on the preceding events (in this case we speak of an Overshooting Continuous Time Random Maxima, abbr. OCTRM). We get the distribution functions of the limit processes and a formula for a Laplace transform for the CTRM and the OCTRM limit. With this formula we have another way to calculate the distribution functions of the limit processes, namely by inversion of the Laplace transform. Moreover we present governing equations, which are in our case time fractional differential equations whose solutions are the distribution functions of our limit processes. Because of the inverse relationship between the CTRM and its first hitting time we get also the Laplace transform of the distribution function of the first hitting time.
\end{abstract}
\maketitle

\section{Introduction}

Classical extreme value theory assumes that observations are collected at regular intervals, or at non random points in time. In some applications, random waiting times between observations are heavy tailed. For waiting times with regularly varying probability tails the mean waiting time can be infinite. In that case, the renewal process that counts the number of observations by time $t$ grows at a sub-linear rate, and the renewal theorem does not apply (see e.g. \cite{feller}, XI.5). 

This paper develops the limiting behavior of the rescaled extremal observation as $t$ tends to infinity. Complementary to \cite{Meerschaert2007a}, where it is assumed that the waiting times and the observations are independent, we allow arbitrary dependence between the $i$th waiting time and the $i$th observation. Silvestrov and Teugels \cite{ST98,ST04} developed a general theory for the joint behavior of sums and maxima in random observation times. However, they do not compute the CDF of the limit process, which is the main result of this paper, see Theorem \ref{Grenzwertverteilungsfunktion} below.
Pancheva and Jordanova \cite{panch04} specifically consider the case of infinite mean waiting times, by adapting arguments from \cite{CTRWinfinite}, along with a powerful transfer theorem (\cite{ST98}, Theorem 3). However in all those papers it is assumed that the waiting times and the observations are at least asymptotically independent. 

Allowing arbitrary dependence leads to various technical problems and the methods used in the above mentioned papers do not apply. Inspired by the well developed theory of (coupled) continuous time random walks, see \cite{Becker-Kern2004,CTRWtriangular,Govfrac}, using recent results on jointly sum/max-stable laws and their domains of attraction, presented in \cite{HeesScheffler1}, we solve this problem completely.

This paper is organized as follows: Section 2 lays out our basic assumptions, defines the processes we want to analyze and recalls some of the results in \cite{HeesScheffler1} needed in the formulation and proof of our main results. Section 3 presents process convergence  of the CTRM and OCTRM processes using a continuous mapping approach. In section 4 we present the main result of this paper: We derive closed formulas for the CDF of the CTRM and OCTRM limit processes at a fixed point $t$ of time and compute the Laplace transform in time of those CDFs. Finally in section 5 we explicitly compute two examples showing the usefulness of the developed theory. 

\section{Problem formulation and basic results}
Let $(W_i,J_i)_{i \in \mathbb{N}}$ be a sequence of iid $\Raum$-valued random variables modeling the $i$th waiting time and the corresponding observation. Observe that we allow arbitrary dependence between $W_i$ and $J_i$. Define \begin{align*}
S(n):=\sum_{i=1}^n W_i \text{ and } M(n):=\bigvee_{i=1}^n J_i
\end{align*} 
which is the time of the $n$-th observation resp. the maximum of the first $n$ observations. The associated partial sum-process S and the partial max-process M are defined by the paths $S(t):=S(\lfloor t \rfloor)$ and $M(t):=M(\lfloor t \rfloor)$. Furthermore we define $N$ to be the renewal process which paths are given by
\begin{align*}
N(t):=\max\left\{n \geq 0: S(n) \leq t \right\}.
\end{align*}
This process counts the number of observations until time $t$. Next we define the main processes of study of this paper.
\begin{defi} We call the process $V$ which is defined by
\begin{align*}
V(t):=M(N(t))=\bigvee_{i=1}^{N(t)}{J_i} 
\end{align*}
{\it Continuous Time Random Maxima} (CTRM). Furthermore we call the process $U$ defined by
\begin{align*}
U(t):=M(N(t)+1)=\bigvee_{i=1}^{N(t)+1}{J_i} 
\end{align*}
{\it overshooting Continuous Time Random Maxima (OCTRM)}.
\end{defi}
The process $V$ is the process which gives you the maximum observation that appears until time $t$. The OCTRM $U$ gives you also the maximum observation, but there you consider one additional jump. That means that in the OCTRM model, the waiting time $W_i$ is dependent of the subsequent jump size $J_i$. In section \ref{limitheorem} we prove a limit theorem for the long time behavior of this two processes. In order to do so, we need to make a natural assumption on the distribution of $(W,J)$. Namely, we assume that  there exist $a_n,b_n>0$ and $d_n\in \mathbb{R}$ such that 
\begin{align} 
(a_nS(n),b_n(M(n)-d_n) \WKonvergenz (D,A) \label{Annahme}
\end{align}
where $\Longrightarrow$ denotes convergence in distribution and $D$ and $A$ are assumed to be non-degenerated. Consequently the random variable $D$ is $\beta$-stable with $0<\beta<1$ and $A$ has an extreme value distribution, hence is Fr\'echet, Weibull or Gumbel distributed. If \eqref{Annahme} holds we say $(W,J)$ belongs to the {\it sum-max-domain of attraction} of $(D,A)$ and that $(D,A)$ is {\it sum-max stable}. A complete characterization of sum-max stable laws and their domains of attraction is presented in the recent paper \cite{HeesScheffler1}. Let us briefly recall some of the notations and results in \cite{HeesScheffler1} needed in the formulation and proofs of our main results in section 4. 

For a probability measure $\mu$ on $\rr_+\times\rr$ let  
\[ \mathcal L(\mu)(s,y)=\int_{\Rplus \times\rr} e^{-st} \ind{\left[-\infty,y\right]} (x) \mu(dt,dx)\]
for $s\geq 0$ and $y\in\rr$. $\mathcal L(\mu)$ is called the {\it CDF-Laplace transform} (C-L transform) of $\mu$. 

Let $F_A$ denote the distribution function of $A$ in \ref{Annahme} and $x_0$ its left endpoint, that is
\begin{align*}
x_0:=\inf\left\{y \in \rr: F_A(y)>0\right\}.
\end{align*} 
By Theorem 3.5 in \cite{HeesScheffler1} we know that the C-L transform of the limit distribution $P_{(D,A)}$ in \eqref{Annahme} can be written as 
\[ \mathcal L(P_{(D,A)})(s,y)=\exp(-\psi(s,y)) \]
for all $s\geq 0$ and $y>x_0$, where $\psi(s,y)$ is called {\it C-L exponent} and given by
\[ \psi(s,y)=\int_{\Rplus}\int_{\left[x_0,\infty\right)}\left(1-e^{-st}\cdot\ind{\left[x_0,y\right]}(x)\right)\Phi(dt,dx) \]
for some $\sigma$-finite measure $\Phi$ on $\rr_+\times(x_0,\infty)$ called the {\it L\'evy-exponent measure} of $(D,A)$. Explicit 
formulas for $\eta$ and $\psi$ are given in Theorem 3.8, Theorem 3.13 and Proposition 3.11 of \cite{HeesScheffler1}.
It follows that 
\begin{equation}\label{levyexpo}
 \Phi_D(dt)=\Phi(dt,\rr)\quad\text{and}\quad \Phi_A(dx)=\Phi(\rr_+,dx) 
\end{equation}
is the L\'evy measure of $D$ and the exponent measure of $A$, respectively.
Especially it is shown in
Corollary 3.10 of \cite{HeesScheffler1} that $D$ and $A$ in \eqref{Annahme} are independent if and only if
\begin{equation}\label{DAind}
\Phi(dt,dx)=\varepsilon_0(dt)\Phi_A(dx)+\varepsilon_0(dx)\Phi_D(dt). 
\end{equation}
Moreover
$\Psi_D(s)=\Psi(s,\infty)$
is the log-Laplace transform of $D$ that is 
\[\mathbb E(e^{-sD})=e^{-\Psi_D(s)} \] and 
$\Psi(0,y)=-\log F_A(y)$.

For the proof of the limit theorem for the long time behavior of our CTRM and OCTRM process under the condition \eqref{Annahme} we need the next result, which tells us that the joint convergence of the sum- and the max-process is equivalent to condition \eqref{Annahme}. It is well known that under the assumption that $W$ belongs to the domain of attraction of a $\beta$-stable random variable the partial sum-process converges in the $J_1$ topology with an appropriate scaling to a $\beta$-stable subordinator, i.e.
\begin{align*}
\left\{a(c)S(ct)\right\}_{t \geq 0} \xlongrightarrow[c \rightarrow \infty]{J_1} \left\{D(t)\right\}_{t \geq 0},
\end{align*}
see for example Theorem 7.1 and Corollary 7.1 in \cite{CTRWinfinite}. Similarly, it is a well known result that under the assumption that $J$ belongs to the max-domain of attraction of an extreme value distributed random variable $A$, that the partial max-process converges in the $J_1$ topology to a $F$-extremal process, i.e.
\begin{align*}
\left\{b(c)(M(ct)-d(c))\right\}_{t \geq 0} \xlongrightarrow[c \rightarrow \infty]{J_1} \left\{A(t)\right\}_{t \geq 0},
\end{align*}
see for example Proposition 4.20 in \cite{Resnick}. The F-extremal process is defined by its finite dimensional distributions
\begin{align*}
P(A(t_i) \leq x_i, 1 \leq i \leq d)=F^{t_1}(\wedge_{i=1}^d x_i) F^{t_2-t_1}(\wedge_{i=2}^d x_i) \cdots F^{t_d-t_{d-1}}(x_d)
\end{align*} 
for $d \geq 1$, times $0=:t_0<t_1< ... <t_d$ and all $x_i \in \mathbb{R}$. Especially it is $P(A(t) \leq x)=F(x)^t$. 
The next Theorem establishes the $J_1$-convergence of the {\it partial (joint) sum-max-process}.
\begin{theorem}\label{TheoremJ1SumMaxProzess} There exist $a_n,b_n>0$ and $d_n \in \mathbb{R}$ such that
\begin{equation}
(a_nS(n),b_n(M(n)-d_n)) \WKonvergenz (D,A), \label{AnnahmefddKonvergenzsatz} 
\end{equation}
if and only if 
\begin{align}
\{\left(a(c)S(ct),b(c)(M(ct)-d(c))\right)\}_{t >0} \xrightarrow[c \rightarrow \infty]{J_1} \{(D(t),A(t))\}_{t>0}, \label{fddKonvegenz}
\end{align}
where $a(c)=a_{[c]}, b(c)=b_{[c]}$ and $d(c)=d_{[c]}$. 
The limit process $\left\{(D(t),A(t))\right\}_{t > 0}$ is uniquely defined by the C-L transforms of its finite-dimensional distributions, which are given by
\begin{align}
&\mathcal{L}(P_{(D(t_j),A(t_j))_{j=1,...,m})})(s,y)=\prod_{j=1}^m \mathcal{L}(P_{(D,A)})^{(t_j-t_{j-1})}\Big(\Sigma_{k=j}^m s_k,\wedge_{k=j}^m y_k \Big),
\end{align}
with $s:=(s_1,...,s_m)\in \Rplus^m$, $y:=(y_1,...,y_m)\in \rr^m$  and $0=:t_0<t_1<...<t_m \in \raum$.
\end{theorem}

\begin{proof}
The proof is similar to the proof in the case $W_i=J_i$ which can be found in \cite{chow1978sum}.
\end{proof}

\section{Limit Theorem for the CTRM}\label{limitheorem}

In this section we prove a limit theorem for the long time behavior of CTRM and OCTRM processes. It is based on the following more general theorem for triangular arrays. The method of proof relies heavily on techniques  developed in \cite{StrakaDis}. For any $c>0$ let $(W_i^{(c)},J_i^{(c)})_{i \in \mathbb{N}}$ be a sequence of iid $\rr_+\times\rr$-valued random vectors. Define the processes $S^{(c)}$ resp. $M^{(c)}$ by their paths $S^{(c)}(t):=S^{(c)}(\lfloor t \rfloor)$ and $M^{(c)}(t):=M^{(c)}(\lfloor t \rfloor)$ where
\begin{align*}
S^{(c)}(n):=\sum_{i=1}^{n} W_i^{(c)} \text{ resp. } M^{(c)}(n):=\bigvee_{i=1}^{n} J_i^{(c)}.
\end{align*} 
Furthermore we define the renewal process $N^{(c)}$ by \[N^{(c)}(t):=\max\left\{n \geq 0: S^{(c)}(n) \leq t \right\}.\] 
Then are
\[ V^{(c)}(t)=M^{(c)}(N^{(c)}(t))\quad\text{\and}\quad U^{(c)}(t)=M^{(c)}(N^{(c)}(t)+1) \]
the corresponding CTRM and OCTRM processes. 

In the following, $\mathbb{D}_{u}(\rr_+)$ denotes the subset of the Skorokhod space $\mathbb{D}(\rr_+)=D([0,\infty),\rr_+)$ of all functions that are unbounded from above and $\mathbb{D}_{\uparrow}(\rr_+)$ and $\mathbb{D}_{\uparrow\uparrow}(\rr_+)$ the subset of all functions $\alpha\in \mathbb{D}(\rr_+)$, with $\alpha$ non-decreasing and strictly  increasing. Furthermore set $\mathbb{D}_{u,\uparrow}(\rr_+):=\mathbb{D}_{u}(\rr_+) \cap \mathbb{D}_{\uparrow}(\rr_+)$ and $\mathbb{D}_{u,\uparrow\uparrow}(\rr_+):=\mathbb{D}_{u}(\rr_+)\cap \mathbb{D}_{\uparrow\uparrow}(\rr_+)$.\\
We denote by $\alpha^{-}$ the left continuous version of a c\`adl\`ag path and by $\alpha^{+}$ the right continuous version of a c\`agl\`ad path $\alpha$. For the purpose of better readability we sometimes also write $\alpha(t-)$ and $\alpha(t+)$ instead of $\alpha^{-}(t)$ and $\alpha^{+}(t)$. Moreover we define the left continuous and right continuous inverse of a path $\alpha \in \mathbb{D}_{u}(\mathbb{R})$  by
\begin{align*}
\alpha^{-1}(t):=\inf\left\{s:\alpha(s) > t\right\} \hspace{1cm} \text{ and } \hspace{1cm} \alpha^{\leftarrow}(t):=\inf\left\{s:\alpha(s) \geq t\right\}.
\end{align*}
An unbounded, increasing function $\mu: \Rplus \rightarrow \mathbb{N}_0$ with $\mu(0)=0$ and jumps of height one is called a discrete time change.

\begin{theorem}\label{allgemeinerGrenzwertsatz}
Let $\mu$ be a discrete time change and $(W_i^{(c)},J_i^{(c)})_{i \in \mathbb{N}}$ for all $c>0$ a sequence of $\Rplus \times \mathbb{R}$-valued random vectors, such that
\begin{align}
\left\{\left(S^{(c)}(t), M^{(c)}(t)\right)\right\}_{t >0} \circ \mu \JKonvergenzc \left\{\left(D(t),A(t)\right)\right\}_{t > 0}, \label{AnnahmealgGrenzwertsatz}
\end{align}
where we assume that the paths of $\left\{D(t)\right\}_{t>0}$ are a.s. strictly monotone increasing. Then
\begin{align}
&\left\{V^{(c)}(t)\right\}_{t > 0} \JKonvergenzc \left\{A(E(t)-)^{+}\right\}_{t>0}  \\
&\hspace{5mm}\text{and }  \nonumber \\
&\left\{U^{(c)}(t)\right\}_{t >0} \JKonvergenzc \left\{A(E(t))\right\}_{t>0}, 
\end{align}
where $E(t):=\inf\left\{s: D(s) > t \right\}$ is the inverse stable subordinator.
\end{theorem}
\begin{proof}
We first look at the case $\mu(t):=\lfloor t \rfloor$. We define $E^{(c)}(t):=(S^{(c)}(t))^{-1}$. From the proof of Proposition 2.4.2 in \cite{StrakaDis} it follows that
\begin{align*}
E^{(c)}(t)=N^{(c)}(t)+1.
\end{align*}
Define the mappings $\Xi : \mathbb{D}_{u,\uparrow}(\mathbb{R}_{+}) \times \mathbb{D}(\mathbb{R}) \rightarrow \mathbb{D}(\mathbb{R})$  and $\Upsilon: \mathbb{D}_{u,\uparrow}(\mathbb{R}_{+}) \times \mathbb{D}(\mathbb{R}) \rightarrow \mathbb{D}(\mathbb{R})$  by
\begin{align}
\Xi(\beta,\sigma):=(\beta^{-} \circ \sigma^{\leftarrow})^{+} \quad\text{ and }\quad \Upsilon(\beta,\sigma):=\beta \circ \sigma^{-1}. \label{Pfadabbildungen}
\end{align}
It follows that
\begin{align*}
\Upsilon \circ (S^{(c)},M^{(c)})(t)=\bigvee_{i=1}^{\lfloor E^{(c)}(t) \rfloor}{J_i^{(c)}}=\bigvee_{i=1}^{N^{(c)}(t) +1} J_i^{(c)}=U^{(c)}(t).
\end{align*}
Furthermore it follows  as in the proof of Proposition 2.4.2 in \cite{StrakaDis} that
\begin{align*}
\Xi \circ (S^{(c)},M^{(c)}) (t)=V^{(c)}(t).
\end{align*}
If we now assume that $\mu$ is an arbitrary discrete time change, there exists a time change $\lambda\in \mathbb{D}(\Rplus)$ such that $\mu(t)=\lfloor \lambda(t) \rfloor$. Hence it follows with Lemma 2.4.1 in \cite{StrakaDis}, that
\begin{align*}
&\Upsilon \circ \left\{\left(S^{(c)}(t), M^{(c)}(t)\right)\right\}_{t >0} \circ \mu = U^{(c)}(t) \text{ and }\\
&\Xi \circ \left\{\left(S^{(c)}(t), M^{(c)}(t)\right)\right\}_{t >0} \circ \mu = V^{(c)}(t).
\end{align*}
Let $P_c$ be the distribution of $\left\{(S^{(c)}(t),M^{(c)}(t))\right\}_{t>0}\circ \mu$ and $P$ the distribution of $\left\{(D(t),A(t))\right\}_{t>0}$. Assumption \eqref{AnnahmealgGrenzwertsatz} is equivalent to $P_{c} \wKonvergenz P$ in $(\mathbb{D}(\Rplus \times \mathbb{R}),J_1)$ as $c \rightarrow \infty$. Let $\mathbb{D}_{u,\uparrow}:=\mathbb{D}_{u,\uparrow}(\mathbb{R}_{+}) \times \mathbb{D}(\mathbb{R})$. In Lemma 2.2.1 in \cite{StrakaDis} it is shown that $\mathbb{D}_{u,\uparrow}$ is Borel-measurable. Furthermore $P_c(\mathbb{D}_{u,\uparrow})=P(\mathbb{D}_{u,\uparrow})=1$. We denote by $P_{c |\mathbb{D}_{u,\uparrow}}$ and $P_{|\mathbb{D}_{u,\uparrow}}$ the restrictions $P_c$ and $P$ on $\mathbb{D}_{u,\uparrow}$, respectively. Due to Corollary 3.3.2 in \cite{EthierKurtz} we have \[P_{c |\mathbb{D}_{u,\uparrow}} \xrightarrow[c \rightarrow \infty]{w} P_{|\mathbb{D}_{u,\uparrow}},\] where $\mathbb{D}_{u,\uparrow}$ is equipped with the relative topology. In Proposition 2.3.8 in \cite{StrakaDis} it is shown further that $\mathbb{D}_{u,\uparrow \uparrow}$ belongs to the set of continuities of $\Xi$ and $\Upsilon$ and in Lemma 2.3.5 in \cite{StrakaDis} that these two mappings are measurable. Let $\textnormal{Disc}(\Xi)$ and $\textnormal{Disc}(\Upsilon)$ denote the set of discontinuities of $\Xi$ and $\Upsilon$, respectively. Due to the assumption that the subordinator $\left\{D(t)\right\}_{t>0}$ has a.s. strictly increasing paths we have $P(\mathbb{D}_{u,\uparrow \uparrow})=1$. Consequently we get $P(\textnormal{Disc}(\Xi))=P(\textnormal{Disc}(\Upsilon))=0$. Using the Continuous Mapping Theorem it then follows that
\begin{align*}
\Xi( P_{c |\Dpu}) \wKonvergenzcinfty \Xi (P_{|\Dpu}) \hspace{1cm} \text{ and } \hspace{1cm} \Upsilon(P_{c|\Dpu}) \wKonvergenzcinfty \Upsilon( P_{|\Dpu})
\end{align*}
and this is equivalent to the assertion. This concludes the proof.
\end{proof}

Using Theorem \ref{allgemeinerGrenzwertsatz} above, we are now able to prove the following limit theorem for the long time behavior of the CTRM and the OCTRM. If the waiting times between the jumps have a finite mean, it is a classical result of renewal theory that the renewal process is asymptotically equivalent to a multiple of the time variable, i.e.
\begin{align*}
N(t) \sim \frac{t}{E(W)} \text{ as } t \rightarrow \infty \text{ a.s.}
\end{align*}
 As a consequence, the appropriate scaled CTRM (resp. OCTRM) behaves asymptotically like a classical extremal process. The interesting case is if the waiting times between the jumps have an infinite mean. This is the case if we assume that the waiting times $W_i\overset{d}{=}W$ are in the domain of attraction of a $\beta$-stable distribution for some $0<\beta<1$. 
 
\begin{theorem}\label{Grenzwertsatz}\ \\
Let $(W_i,J_i)_{i \in \mathbb{N}},(W,J)$ be iid and $\Rplus \times \mathbb{R}$-valued random vectors. We assume that there exist $a_n,b_n>0$ and $d_n \in \mathbb{R}$ such that
\begin{align}\label{AnnahmeGrenzwertsatz}
(a_n S(n), b_n(M(n)-d_n)) \WKonvergenz (D,A),
\end{align}
where $D$ is strictly $\beta$-stable with $0<\beta<1$ and $A$ has an extreme value distribution. Then there exist functions $\tilde{b}(c)$ and $\tilde{d}(c)$ such that
\begin{align}
&\left\{\tilde{b}(c)\big(V(ct)-\tilde{d}(c)\big)\right\}_{t>0} \JKonvergenzc \left\{A(E(t)-)^{+}\right\}_{t>0} \label{CTRMGrenzwert}\\
&\hspace{5mm}\text {and } \nonumber \\ 
&\left\{\tilde{b}(c)\big(U(ct)-\tilde{d}(c)\big)\right\}_{t>0} \JKonvergenzc \big\{A(E(t))\big\}_{t>0}.\label{OCTRMGrenzwert}
\end{align}
Here $\left\{A(t)\right\}_{t>0}$ is a F-extremal process with $P(A(t) \leq x)=F(x)^t$, where $F$ is the distribution function of $A$. Furthermore $\left\{E(t)\right\}_{t>0}$ is the (left continuous) inverse of the stable subordinator $\left\{D(t)\right\}_{t>0}$.
\end{theorem}

\begin{proof}
In view of Theorem \ref{TheoremJ1SumMaxProzess} we know that there exists a function $a(c)$ which is regularly varying with index $-1/\beta$  and functions $b(c)$ and $d(c)$ such that
\begin{eqnarray*}
\left\{(a(c)S(ct),b(c)(M(ct)-d(c))\right\}_{t>0} & \JKonvergenzc \left\{(D(t),A(t))\right\}_{t>0}. 
\end{eqnarray*}
Since $a$ is regularly varying with index $-1/\beta$,  $1/a$ is regularly varying with index $1/\beta$. Hence there exists a function $\tilde{a}(c)$ regularly varying with index  $\beta$ such that $1/a(\tilde{a}(c)) \sim c$ as $c \rightarrow \infty$. (cf. p.738 in \cite{Becker-Kern2004} or Property 1.5.5. in \cite{Seneta1976}).
Let $g_c: \mathbb{D}_{\uparrow,u}(\Rplus) \times \mathbb{D}(\mathbb{R}) \rightarrow \mathbb{D}_{\uparrow,u}(\Rplus) \times \mathbb{D}(\mathbb{R})$ 
be defined as $g_c(x,y)=((ca(\tilde{a}(c))^{-1}x,y)$. Then $g_c(x,y) \rightarrow (x,y)$ as $c\rightarrow \infty$ in $(\mathbb{D}_{\uparrow,u}(\Rplus) \times \mathbb{D}(\mathbb{R}),J_1)$. It follows with the generalized continuous-mapping theorem (see for example Theorem 3.4.4. in \cite{Whitt}) that
\begin{align*}
\{\big(c^{-1}S(\tilde{a}(c)t), b(\tilde{a}(c))(M(\tilde{a}(c)t)-d(\tilde{a}(c)))\big)\}_{t>0} \JKonvergenzc \left\{(D(t),A(t))\right\}_{t>0}.
\end{align*}
Using Theorem \ref{allgemeinerGrenzwertsatz}  with $W_i^{(c)}:=c^{-1}W_i$ and $J_i^{(c)}:=b(\tilde{a}(c))(J_i-d(\tilde{a}(c)))$, $\mu(t):=\lfloor \tilde{a}(c) t \rfloor$ and since  $N^{(c)}(t)=N(ct)$ if follows that
\begin{align}
\big\{b(\tilde{a}(c)) \bigvee_{i=1}^{N(ct)} (J_i-d(\tilde{a}(c)))\big\}_{t >0} & \JKonvergenzc \left\{A(E(t)-)^{+}\right\}_{t >0}. \label{GrenzwertCTRW} 
\end{align}
The proof for the OCTRM is similar.
\end{proof}

\section{Law of the CTRM and OCTRM scaling limit}
In this section we derive the distribution functions of the limit processes obtained in Theorem \ref{Grenzwertsatz}. The next theorem provides two ways to calculate the distribution function of the CTRM and the OCTRM long time limit at a fixed point $t>0$ of time. On one hand we get a closed formula for calculating the distribution function based on the joint distribution of $(D(t),A(t))$ and the L\'evy measure $\Phi$ of $(D,A)$. However, the joint distribution of $(D(t),A(t))$ is only in a few cases explicitly given.
On the other hand we obtain a formula for a Laplace transform in time $t$ of the CDFs which can be used to calculate the distribution functions by inversion of the Laplace transform.  This method will be used in the examples presented in section 5.\\
In the following let $x_F$ be the right endpoint of the distribution function $F$ of the extreme value distributed random variable $A$, i.e.
\begin{align*}
x_F:=\sup \left\{x \in\rr : F_A(x)<1 \right\}.
\end{align*}

\begin{theorem}\label{Grenzwertverteilungsfunktion}
(a) For a fixed $t>0$ the distribution function of the CTRM limit in \eqref{CTRMGrenzwert} in Theorem \ref{Grenzwertsatz} is given by
\begin{align} \label{VerteilungsfunktionCTRM}
G_t(x):&=P\left(A(E(t)-)^{+} \leq x \right)  \\
&= \int_{s=0}^{\infty} \int_{u\in\mathbb{R}} \int_{\tau=0}^t \ind{\left[-\infty,x\right]}(u) \Phi_D(t-\tau,\infty) P_{(D(s),A(s))}(d\tau,du) ds. \nonumber
\end{align}
Furthermore for an arbitrary $\xi>0$ and for all $x_0<x<x_F$ we have the following Laplace transform:
\begin{align}
\intnu e^{-\xi t} P\left(A(E(t)-)^{+} \leq x \right) dt &= \frac{1}{\xi}\frac{\Psi_D(\xi)}{\Psi(\xi,x)}. \label{LTderVF}
\end{align}
(b) For a fixed $t>0$ the distribution function of the OCTRM limit in \eqref{OCTRMGrenzwert} in Theorem \ref{Grenzwertsatz} is given by
\begin{align}\label{VerteilungsfunktionOCTRM}
F_t(x):&=P\left(A(E(t)) \leq x \right) \\ 
&=\int_{s=0}^{\infty} \int_{u\in\mathbb{R}} \int_{\tau=0}^t \ind{[-\infty,x]}(u)\Phi ((t-\tau,\infty),[-\infty,x]) P_{(D(s),A(s))}(d\tau,du)ds. \nonumber
\end{align}
Furthermore for an arbitrary $\xi>0$ and for all $x_0<x<x_F$ we have the following Laplace transform:
\begin{eqnarray}
\intnu e^{-\xi t} P\left(A(E(t)) \leq x \right) dt = \frac{1}{\xi} \frac{\Psi(\xi, x)+\log F_A(x)}{\Psi(\xi, x)}. \label{LTderVFOCTRW}
\end{eqnarray}
\end{theorem} 

The proof is based on a series of Lemmas presented in the following.

\begin{lemma} \label{lemmaCTRM1} 
For an arbitrary $x>x_0$ and $\xi>0$ it is
\begin{align}
&\intnu e^{-\xi t} \Phi((t,\infty),[-\infty,x]) dt = \frac{1}{\xi} \left(\Psi(\xi,x)+\log F_{A}(x)\right),
\end{align}
where $\Psi$ is the C-L exponent of $(D,A)$.
\begin{proof}
For an arbitrary $x>x_0$ and $\xi>0$ we get with Fubini's Theorem
\begin{align*}
&\int_{t=0}^{\infty} e^{-\xi t} \Phi((t,\infty),[-\infty,x]) dt\\
&\hspace{5mm}=\int_{t=0}^{\infty} \int_{y\in \mathbb{R}} \int_{u=0}^{\infty} e^{-\xi t} \ind{\left[-\infty,x\right]} (y)  \ind{(t,\infty)}(u) \Phi(du,dy) dt\\
&\hspace{5mm}=\int_{u=0}^{\infty} \int_{y \in \mathbb{R}} \ind{\left[-\infty,x\right]} (y) \Big(\int_{t=0}^{\infty} e^{-\xi t}\ind{(t,\infty)}(u) dt \Big) \Phi(du,dy).
\end{align*}
Furthermore we obtain due to $\ind{\left[-\infty,x\right]}(y)=1-\ind{(x,\infty)}(y)$ and the definition of the C-L exponent 
\begin{align*}
&\hspace{5mm}= \frac{1}{\xi} \int_{u=0}^{\infty} \int_{y\in\mathbb{R}}  \left( 1-\ind{\left[-\infty,x\right]}(y)e^{-\xi u}-\ind{(x,\infty)}(u) \right) \Phi(du,dy)\\
&\hspace{5mm}=\frac{1}{\xi}\Big[\Psi(\xi,x)-\int_{y\in\mathbb{R}} \ind{(x,\infty)}(y) \Phi(\Rplus,dy) \Big]\\
&\hspace{5mm}=\frac{1}{\xi}\Big[\Psi(\xi,x)+\log F_A(x) \Big].
\end{align*}
\end{proof}
\end{lemma}

\begin{lemma}\label{LemmazuGtundFt}
The functions $G_t$ and $F_t$ defined in \eqref{VerteilungsfunktionCTRM} and \eqref{VerteilungsfunktionOCTRM} are distribution functions. Furthermore, for fixed $x\in\rr$ the mappings $t \mapsto G_t(x)$ and $t \mapsto F_t(x)$ for $t>0$ are right-continuous. Additionally for an arbitrary $\xi>0$ and any $x_0<x<x_F$ we have
\begin{align}\label{LTvonGt}
\intnu e^{-\xi t} G_t(x) dt &= \frac{1}{\xi}\frac{\Psi_D(\xi)}{\Psi(\xi,x)}  
\end{align}
and
\begin{align}\label{LTvonFt}
\intnu e^{-\xi t} F_t(x) dt = \frac{1}{\xi} \frac{\Psi(\xi, x)+\log F_A(x)}{\Psi(\xi, x)}.
\end{align}
\end{lemma}

\begin{proof}
We first show that $G_t$ and $F_t$ are distribution functions.  For an arbitrary $h>0$ we get on the one hand that
\[\lim_{h \downarrow 0} \ind{\left[-\infty,x+h\right]}(u)=\ind{\left[-\infty,x\right]}(u)\]
and on the other hand we we have 
\begin{align*}
\left|\Phi_D(t-\tau,\infty)\ind{\left[-\infty,x+h\right]}(u)\right| \leq \Phi_{D}(t-\tau,\infty).
\end{align*}
Since
\begin{align*} 
&\int_{s=0}^{\infty} \int_{u \in \mathbb{R}} \int_{\tau=0}^{t} \Phi_D(t-\tau,\infty) P_{(D(s),A(s))}(d\tau,du) ds\\
&\hspace{5mm}=\int_{s=0}^{\infty} \int_{\tau=0}^{t} \Phi_D(t-\tau,\infty) P_{D(s)} (d\tau) ds\\
&\hspace{5mm}= 1
\end{align*}
where we used in the last line Corollary 6.2 in \cite{Kesten}, it follows with  dominated convergence  that \[\lim_{h \downarrow 0} G_t(x+h)=G_t(x).\] 
Similarly we can show with the dominated convergence  that $\lim_{x \rightarrow \infty} G_t(x)=1$
and ${\lim_{x \rightarrow -\infty} G_t(x)=0}$. Monotonicity is obvious. That $F_t(x)$ is right-continuous follows likewise with the dominated convergence  since 
\begin{align*}
\ind{[-\infty,x]}(u)\Phi((t-\tau,\infty),[-\infty,x]) \leq \Phi_D(t-\tau,\infty)
\end{align*}
and hence we get as above that $\lim_{h \downarrow 0} F_t(x+h)=F_t(x)$. With the same argument it follows that ${\lim_{x \rightarrow \infty} F_t(x)=1}$, ${\lim_{x \rightarrow -\infty} F_t(x)=0}$ as well. Monotonicity of $F_t(x)$ is again obvious.\\
Next we show that for a fixed $x\in\rr$ the mapping $t \rightarrow G_t(x)$ is right-continuous. For an arbitrary $h>0$ we get
\begin{align*}
&|G_t(x)-G_{t+h}(x)| \\
&\hspace{5mm}=\int_{s=0}^{\infty} \int_{u \in \mathbb{R}} \int_{\tau=0}^{t} \Big[ \Phi_D(t-\tau,\infty)-\Phi_D(t+h-\tau,\infty)\Big] \ind{\left[-\infty,x\right]}(u) P_{(D(s),A(s))}(d\tau,du) ds\\
&\hspace{10mm}-\int_{s=0}^{\infty} \int_{u \in \mathbb{R}} \int_{\tau=t}^{t+h} \Phi_D(t+h-\tau,\infty) \ind{\left[-\infty,x\right]}(u) P_{(D(s),A(s))}(d\tau,du) ds\\
&\hspace{5mm}=:I_{h}^{(1)}-J_h^{(1)}.
\end{align*}
We first consider $I_{h}^{(1)}$ and obtain as in the proof of Theorem 3.1 in \cite{CTRWtriangular} that $I_{h}^{(1)}$ converges to $0$ as $h \downarrow 0$. In fact
\begin{align*}
I_{h}^{(1)} \leq \int_{s=0}^{\infty} \int_{\tau=0}^{t} \Big[ \Phi_D(t-\tau,\infty)-\Phi_D(t+h-\tau,\infty)\Big]  P_{D(s)}(d\tau) ds.
\end{align*}
Due to the fact that $v \mapsto \Phi_D(v,\infty)$ is right-continuous, we get as $h \downarrow 0$ 
\begin{align*}
\Phi_D(t-\tau,\infty)-\Phi_D(t+h-\tau,\infty) \longrightarrow 0,
\end{align*}
for all $0 \leq \tau \leq t$. Additionally we have
\begin{align*}
&\Phi_D(t-\tau,\infty)-\Phi_D(t+h-\tau,\infty) \leq \Phi_D(t-\tau,\infty)
\end{align*}
and
\begin{align*}
\int_{s=0}^{\infty} \int_{\tau=0}^t \Phi_D(t-\tau,\infty) P_{D(s)} (d\tau) ds
&=1, 
\end{align*}
due to Corollary 6.2 in \cite{Kesten}
it then follows that $I_h^{(1)} \rightarrow 0$ as $h\downarrow 0$. It remains to show that $J_{h}^{(1)}\rightarrow 0$ as $h \downarrow 0$. This also follows along the lines of proof of Theorem 3.1 in \cite{CTRWtriangular}. We have
\begin{align*}
J_{h}^{(1)} & \leq \int_{s=0}^{\infty}  \int_{\tau=t}^{t+h} \Phi_D (t+h-\tau,\infty) P_{D(s)}(d\tau) ds\\
&=\int_{s=0}^{\infty}\int_{\tau=0}^{t+h} \Phi_D(t+h-\tau,\infty) P_{D(s)}(d\tau) ds-\int_{s=0}^{\infty}\int_{\tau=0}^{t} \Phi_D(t+h-\tau,\infty)P_{D(s)}(d\tau) ds\\
&=1-\int_{s=0}^{\infty}\int_{\tau=0}^{t} \Phi_D(t+h-\tau,\infty) P_{D(s)}(d\tau).
\end{align*}
It then follows that
\begin{align*}
1-\int_{s=0}^{\infty}\int_{\tau=0}^{t} \Phi_D(t+h-\tau,\infty) P_{D(s)}(d\tau) \rightarrow  0
\end{align*}
as $h \downarrow 0$. Hence $t \mapsto G_t(x)$ is right-continuous. We now show that $t \mapsto F_t(x)$ is right-continuous, too. For $t >0$ and $h>0$ we get
\begin{align*}
&|F_t(x)-F_{t+h}(x)|\\
&\hspace{5mm}=\int_{s=0}^{\infty}\int_{u \in \mathbb{R}}\int_{\tau=0}^{t} \ind{[-\infty,x]}(u)\big[\Phi((t-\tau,\infty),[-\infty,x])\\
&\hspace{10mm}-\Phi((t+h-\tau,\infty),[-\infty,x])\big]P_{(D(s),A(s))}(d\tau, du) ds\\
&\hspace{10mm}-\int_{s=0}^{\infty}\int_{u \in \mathbb{R}} \int_{\tau=t}^{t+h} \ind{[-\infty,x]} \Phi\big((t+h-\tau,\infty),\left[-\infty,x \right]\big) P_{(D(s),A(s))}(d\tau,du) ds\\
&\hspace{5mm}=:I_h^{(2)} - J_h^{(2)}.
\end{align*}
It follows as for $I_h^{(1)}$ and $J_h^{(1)}$ above, that as $h \downarrow 0$
\begin{align*}
I_h^{(2)} &\leq \int_{s=0}^{\infty} \int_{\tau=0}^{t} \Phi \big((t-\tau,\infty),\left[-\infty,x \right]\big)-\Phi\big((t+h-\tau,\infty),\left[-\infty,x\right]\big)P_{D(s)}(d\tau) ds \rightarrow 0\\
\hspace{10mm}\text{and }&\\
J_h^{(2)} &\leq \int_{s=0}^{\infty} \int_{\tau=t}^{t+h} \Phi\big((t+h-\tau,\infty),\left[-\infty,x \right]\big)P_{D(s)}(d\tau) ds \rightarrow 0.
\end{align*}
Consequently $t \mapsto F_t(x)$ is also right-continuous. Next we show \eqref{LTvonGt}. Compute
\begin{align}
&\int_{t=0}^{\infty} e^{-\xi t} G_t(x) dt \nonumber\\
&\hspace{5mm}= \int_{t=0}^{\infty} \int_{s=0}^{\infty} \int_{u\in\mathbb{R}} \int_{\tau=0}^t  e^{-\xi t}\Phi_D(t-\tau,\infty) \ind{\left[-\infty,x\right]}(u) P_{(D(s),A(s))}(d\tau,du) ds\, dt \nonumber\\
&\hspace{5mm}=\int_{s=0}^{\infty}\int_{u \in \mathbb{R}}\int_{\tau=0}^{\infty} \ind{\left[-\infty,x\right]}(u) \left( \int_{t=0}^{\infty} e^{-\xi t} \ind{\left[0,t\right]}(\tau) \Phi_D(t-\tau,\infty) dt \right) P_{(D(s),A(s))}(d\tau,du) ds \nonumber  \\
&\hspace{5mm}=\int_{s=0}^{\infty}\int_{u \in \mathbb{R}}\int_{\tau=0}^{\infty} \ind{\left[-\infty,x\right]}(u) \left( \int_{t=\tau}^{\infty} e^{-\xi t} \Phi_D (t-\tau,\infty) dt \right) P_{(D(s),A(s))}(d\tau,du) ds. \label{BeweisGrenzverteilung1}
\end{align}
Since
\begin{align*}
\int_{t=\tau}^{\infty} e^{-\xi t} \Phi_D(t-\tau,\infty)dt
&= \int_{z=0}^{\infty} \int_{t=0}^{\infty} e^{-\xi t} \ind{(\tau,\infty)}(t)\ind{(t-\tau,\infty)}(z) dt \, \Phi_D(dz)\\
&=\int_{z=0}^{\infty} \int_{t=0}^{\infty} e^{-\xi t} \ind{(\tau,z+\tau)}(t) dt \, \Phi_D (dz) \\
&=\frac{1}{\xi} e^{-\xi \tau} \int_{z=0}^{\infty} (1-e^{-\xi z}) \Phi_D(dz)\\
&=\frac{1}{\xi}e^{-\xi \tau} \Psi_D(\xi),
\end{align*}
and inserting this in \eqref{BeweisGrenzverteilung1}, we obtain
\begin{align*}
&\intnu e^{-\xi t} G_t(x) dt\\
&\hspace{5mm}=\frac{\Psi_D(\xi)}{\xi} \int_{s=0}^{\infty} \left( \int_{u \in \mathbb{R}} \int_{\tau=0}^{\infty} e^{-\xi \tau} \ind{\left[-\infty,x\right]}(u) P_{(D(s),A(s))}(d\tau,du)\right) ds\\
&\hspace{5mm}=\frac{\Psi_D(\xi)}{\xi} \int_{s=0}^{\infty} \mathcal{L} \left(P_{(D(s),A(s))}\right)(\xi,x) ds\\
&\hspace{5mm}=\frac{\Psi_D(\xi)}{\xi} \int_{s=0}^{\infty} \exp(-s \Psi(\xi,x)) ds\\
&\hspace{5mm}=\frac{1}{\xi}\frac{\Psi_D(\xi)}{\Psi(\xi,x)}.
\end{align*}
Similarly we can show \eqref{LTvonFt}, because with change of variables $v=t-\tau$ we have
\begin{align*}
&\int_{t=0}^{\infty} e^{- \xi t} F_t(x) dt\\
&\hspace{3mm}=\int_{t=0}^{\infty}\int_{s=0}^{\infty} \int_{u \in \mathbb{R}} \int_{\tau=0}^{t} e^{-\xi t} \ind{\left[-\infty,x\right]}(u)\Phi ((t-\tau,\infty),[-\infty,x]) P_{(D(s),A(s))} (d\tau,du) ds\ dt\\
&\hspace{3mm}=\int_{s=0}^{\infty} \int_{u \in \mathbb{R}} \int_{\tau=0}^{\infty} \int_{t=\tau}^{\infty} e^{-\xi t} \ind{\left[-\infty,x\right]}(u)\Phi ((t-\tau,\infty),[-\infty,x]) dt P_{(D(s),A(s))} (d\tau,du) ds\\
&\hspace{3mm}=\int_{s=0}^{\infty} \int_{u \in \mathbb{R}} \int_{\tau=0}^{\infty} \int_{v=0}^{\infty} e^{-\xi(v+\tau)} \ind{\left[-\infty,x\right]}(u)\Phi ((v,\infty),[-\infty,x]) dv P_{(D(s),A(s))} (d\tau,du) ds\\
\end{align*}
and with Lemma \ref{lemmaCTRM1} it follows that
\begin{align*}
&\int_{t=0}^{\infty} e^{- \xi t} F_t(x) dt\\
&\hspace{5mm}=\frac{1}{\xi} \Big(\Psi(\xi, x)+\log F_A(x)\Big) \int_{s=0}^{\infty} \Big( \int_{\tau=0}^{\infty} \int_{u \in \mathbb{R}} \ind{\left[-\infty,x\right]}(u) e^{-\xi \tau} P_{(D(s),A(s))} (d\tau,du) \Big) ds\\
&\hspace{5mm}=\frac{1}{\xi} \Big(\Psi(\xi, x)+\log F_A(x)\Big) \int_{s=0}^{\infty} \exp({-s \Psi(\xi, x)}) ds\\
&\hspace{5mm}=\frac{1}{\xi} \frac{\Psi(\xi, x)+\log F_A(x)}{\Psi(\xi, x)}.
\end{align*}
\end{proof}

In the following let $x_r$ and $x_l$ denote the left resp. the right endpoint of the distribution function $F_J$ of $J$, i.e.
\[x_{r}:=\sup \left\{x: F_{J}(x)<1\right\} \text{ and } x_{l}:=\inf \left\{x: F_{J}(x)>0\right\}.\]
Moreover, let $\tilde P_W(s)=\mathbb E[e^{-sW}]$, $s\geq 0$ denote the Laplace transform of $W$.

\begin{lemma} \label{lemmaLaplaceCTRMundOCTRM} 
\ \\
\begin{itemize}
\item[(a)] For the \textbf{CTRM process} $\left\{V(t)\right\}_{t>0}=\left\{M(N(t))\right\}_{t>0}$  we have for all $\xi>0$ and all $x_{l}<x<x_r$ that
\begin{eqnarray*}
\int_{0}^{\infty} e^{-\xi t}P\left(V(t) \leq x \right) dt =\frac{1}{\xi} \frac{1-\tilde{P}_W(\xi)}{1-\mathcal{L}(P_{(W,J)})(\xi,x)}.
\end{eqnarray*}
\item[(b)] For the \textbf{OCTRM process} $\left\{U(t)\right\}_{t>0}=\left\{M(N(t)+1)\right\}_{t>0}$ we have for all $\xi>0$ and all $x_{l}<x<x_r$ that
\begin{eqnarray*}
\int_{0}^{\infty} e^{-\xi t}P\left(U(t) \leq x \right) dt=\frac{1}{\xi}\frac{F_J(x)-\mathcal{L}\left(P_{(W,J)}\right)(\xi,x)}{1-\mathcal{L}\left(P_{(W,J)}\right)(\xi,x)}.
\end{eqnarray*}
\end{itemize}
\end{lemma}

\begin{proof}
First we will show (a). Since $\left\{N(t) \geq n \right\}=\left\{S(n) \leq t \right\}$ we get
\begin{align}
\begin{split}\label{BeweisLTCTRM1}
&\intnu e^{-\xi t} P\left(M(N(t)) \leq x\right)dt\\
&\hspace{5mm}=\sumnu \intnu e^{-\xi t} P \left(M(N(t)) \leq x, N(t)=n \right) dt \\
&\hspace{5mm}= \sumnu \intnu e^{-\xi t} \Big[ P\left(M(n) \leq x, N(t) \geq n \right) - P \left(M(n) \leq x, N(t) \geq n+1 \right)\Big] dt\\
&\hspace{5mm}= \sumnu \intnu e^{-\xi t} \Big[ P\left(M(n) \leq x, S(n) \leq t \right)-P\left(M(n) \leq x, S(n+1) \leq t \right) \Big] dt.
\end{split}
\end{align}
Observe that in view of Proposition 2.3 in \cite{HeesScheffler1} we have
\begin{align}
\begin{split}\label{BeweisLTCTRM2}
&\intnu e^{-\xi t} P\left(M(n) \leq x, S(n) \leq t \right) dt 
=\intnu \int \indK{M(n) \leq x}  e^{-\xi t} \ind{\left\{S(n) \leq t\right\}} dt \, dP\\
&=\int  \indK{M(n) \leq x} \frac{1}{\xi} e^{-\xi S(n)}dP
=\frac{1}{\xi} \mathcal{L} \left(P_{(S(n),M(n))}\right)(\xi,x)\\
&=\frac{1}{\xi} \left(\mathcal{L}\left(P_{(W,J)}\right)(\xi,x)\right)^n.
\end{split}
\end{align}
Changing the order of integration and with a change of variables we obtain
\begin{align}\label{BeweisLTCTRMNeu}
\begin{split}
&\intnu e^{-\xi t} P\left(M(n) \leq x, S(n+1)\leq t \right) dt\\
&\hspace{5mm}=\intnu e^{-\xi t} P\left(M(n) \leq x, S(n)+W_{n+1} \leq t \right) dt\\
&\hspace{5mm}= \intnu\int_{0}^{t}  e^{-\xi t} P\left(M(n) \leq x, S(n) \leq t-\tau\right) P_W(d \tau) dt \\
&\hspace{5mm}= \intnu \int_{\tau}^{\infty} e^{-\xi t} P\left(M(n) \leq x, S(n) \leq t-\tau\right)  dt P_W(d \tau)\\ 
&\hspace{5mm}=\intnu \intnu e^{-\xi(\tau + u)} P\left(M(n) \leq x, S(n) \leq u \right) du P_W(d \tau)  \\
&\hspace{5mm}= \left(  \intnu e^{-\xi \tau} P_W(d\tau)\right)\left( \intnu e^{-\xi u} P\left(M(n) \leq x, S(n) \leq u \right) du \right)\\
&\hspace{5mm}= \tilde{P}_W(\xi) \frac{1}{\xi} \left( \mathcal{L} \left(P_{(W,J)}\right)(\xi,x)\right)^n.
\end{split}
\end{align}
If we put \eqref{BeweisLTCTRM2} and \eqref{BeweisLTCTRMNeu} in \eqref{BeweisLTCTRM1} we receive
\begin{align*}
&\intnu e^{-\xi t} P\left(M(N(t)) \leq x \right) dt\\
&\hspace{5mm}=\sumnu{\left(\frac{1}{\xi}\left(\mathcal{L}\left(P_{(W,J)}\right)(\xi,x)\right)^n-\tilde{P}_W(\xi)\frac{1}{\xi}\left(\mathcal{L}\left(P_{(W,J)}\right)(\xi,x)\right)^n\right)}\\
&\hspace{5mm}=\frac{1}{\xi} \left(1-\tilde{P}_W(\xi)\right)  \sumnu \left(\mathcal{L}\left(P_{(W,J)}\right)(\xi,x)\right)^n\\
&\hspace{5mm}=\frac{1-\tilde{P}_W(\xi)}{\xi}\frac{1}{1-\mathcal{L}\left(P_{(W,J)}\right)(\xi,x)}.
\end{align*}
Now we will prove (b). As in the proof of (a) we have
\begin{align}\label{BeweisLTCTRM4}
&\intnu e^{-\xi t} P\left( M(N(t)+1)\leq x\right) dt\\
&\hspace{5mm}=\sumnu \intnu e^{-\xi t} \Big[P\left(M(N(t)+1) \leq x, N(t)=n \right)\Big]dt \nonumber\\
&\hspace{5mm}=\sumnu \intnu e^{-\xi t}\Big[P\left(M(n+1) \leq x, N(t)\geq n \right)-P\left(M(n+1) \leq x, N(t)\geq n+1 \right)\Big] \nonumber\\
&\hspace{5mm}=\sumnu \intnu e^{-\xi t} \Big[P\left(M(n+1) \leq x, S(n)\leq t \right)-P\left(M(n+1) \leq x, S(n+1)\leq t \right)\Big]. \nonumber
\end{align}
The first part of the integral simplifies to
\begin{align}
\begin{split}\label{BeweisLTCTRM5}
&\intnu e^{-\xi t} P\left(M(n+1) \leq x, S(n) \leq t\right)dt=\int \indK{M(n+1) \leq x} \intnu e^{-\xi t} \indK{S(n) \leq t}dt dP\\
&\hspace{5mm}=\int \indK{M(n+1) \leq x} \frac{1}{\xi}e^{-\xi S(n)} dP=P\left(J_{n+1}\leq x\right) \int \indK{M(n)\leq x} \frac{1}{\xi}e^{-\xi S(n)} dP \\
&\hspace{5mm}=F_{J}(x)\frac{1}{\xi}\left(\mathcal{L}\left(P_{(W,J)}\right)(\xi,x)\right)^n.
\end{split}
\end{align}
For the second part of the integral we get
\begin{align}
\begin{split}\label{BeweisLTCTRM6}
&\intnu e^{-\xi t} P\left(M(n+1) \leq x, S(n+1) \leq t\right) dt=\int \indK{M(n+1)\leq x} \int_{S(n+1)}^{\infty} e^{-\xi t} dt dP\\
&\hspace{5mm}=\int \indK{M(n+1)\leq x}\frac{1}{\xi}e^{-\xi S(n+1)} dP=\frac{1}{\xi}\left(\mathcal{L}\left(P_{\left(W,J\right)}\right)(\xi,x)\right)^{n+1}.
\end{split}
\end{align}
If we put \eqref{BeweisLTCTRM5} and \eqref{BeweisLTCTRM6} in \eqref{BeweisLTCTRM4} we obtain
\begin{align*}
&\intnu e^{-\xi t} P\left(M(N(t)+1)\leq x\right)dt\\
&\hspace{5mm}=\sumnu F_J(x) \frac{1}{\xi}\left(\mathcal{L}\left(P_{(W,J)}\right)(\xi,x)\right)^n-\frac{1}{\xi}\left(\mathcal{L}\left(P_{\left(W,J\right)}\right)(\xi,x)\right)^{n+1}\\
&\hspace{5mm}=\frac{1}{\xi}\left( F_J(x)-\mathcal{L}\left(P_{\left(W,J\right)}\right)(\xi,x)\right) \sumnu \left(\mathcal{L}\left(P_{(W,J)}\right)(\xi,x)\right)^n\\
&\hspace{5mm}=\frac{1}{\xi}\frac{ F_J(x) -\mathcal{L}\left(P_{(W,J)}\right)(\xi,x)}{1-\mathcal{L}\left(P_{(W,J)}\right)(\xi,x)}
\end{align*}
and the proof is complete.
\end{proof}

\begin{lemma} \label{KonvergenzCLTransformierte}\ \\
Under the assumptions of Theorem \ref{Grenzwertsatz} we have:
\begin{itemize}
\item[(a)]For the \textbf{CTRM process} $\left\{V(t)\right\}_{t>0}$ we have for all $x_0<x<x_F$ and all $\xi>0$
\begin{align}
&\intnu e^{-\xi t} P\left(\tilde{b}(c)(V(ct)-\tilde{d}(c)) \leq x \right) dt \xrightarrow[c \rightarrow \infty]{} \frac{1}{\xi}\frac{\Psi_D(\xi)}{\Psi(\xi, x)}.
\end{align}
\item[(b)] For the \textbf{OCTRM process} $\left\{U(t)\right\}_{t>0}$ we have for all $x_0<x<x_F$ and all $\xi>0$ 
\begin{align}
&\intnu e^{-\xi t} P\left(\tilde{b}(c)(U(ct)-\tilde{d}(c)) \leq x \right) dt \xrightarrow[c \rightarrow \infty]{} \frac{1}{\xi}\frac{\Psi(\xi, x)+\log F_A(x)}{\Psi(\xi, x)}.
\end{align} 
\end{itemize}

\begin{proof}
We know from the proof of Theorem \ref{Grenzwertsatz} that there exists a regularly varying function $\tilde{a}(c)$ with index $\beta$ with $1/a(\tilde{a}(c)) \sim c$ as $c \rightarrow \infty$, furthermore functions $\tilde{b}(c)$ and $\tilde{d}(c)$, such that
\begin{eqnarray*}
\big(c^{-1}S(\tilde{a}(c)), \tilde{b}(c)(M(\tilde{a}(c))-\tilde{d}(c))\big) \xLongrightarrow[c \rightarrow \infty]{} \big(D,A\big).
\end{eqnarray*}
With the continuity theorem for the C-L transform (see Theorem 3.2 in \cite{HeesScheffler1}), it follows that
\begin{align*}
&\mathcal{L} \big(P_{(c^{-1} S(\tilde{a}(c)), \tilde{b}(c)(M(\tilde{a}(c))-\tilde{d}(c))}\big)(\xi,x) \Konvergenzcinfty \mathcal{L}(P_{(D,A)})(\xi,x)
\end{align*}
for all $\xi \in \Rplus$ and $x\in\rr$. Then it follows from Proposition 2.3 of \cite{HeesScheffler1} that
\begin{align}
\big( \mathcal{L} \big(P_{(c^{-1} W, \tilde{b}(c)(J-\tilde{d}(c)))} \big)(\xi,x) \big)^{\lfloor \tilde{a}(c) \rfloor} \Konvergenzcinfty \exp\left({-\Psi(\xi,x)}\right). \label{LemmaKonvergenzLTCTRM1}
\end{align}
If we now apply the logarithm on each side and use that  $\log z \sim  z-1$ as $z \rightarrow 1$ we get
\begin{align}
\lfloor \tilde{a}(c) \rfloor \left(1-\mathcal{L} \big(P_{(W,J)}\big)\big(\xi c^{-1},x\tilde{b}(c)^{-1}+\tilde{d}(c)\big)\right)\Konvergenzcinfty \Psi(\xi,x) \label{LemmaKonvergenzLTCTRM2}
\end{align}
for all $\xi \in \Rplus$ and $x\in\rr$. If we let $x=\infty$ in \eqref{LemmaKonvergenzLTCTRM2} it follows due to $\Psi(\xi,\infty)=\Psi_D(\xi)$ that
\begin{align}
\lfloor \tilde{a}(c) \rfloor \big(1-\tilde{P}_{W}(\xi c^{-1})\big) \Konvergenzcinfty \Psi_D(\xi). \label{LemmaKonvergenzLTCTRMGl1}
\end{align}
If we set $\xi=0$ it follows since $\Psi(0,x)=-\log F_A(x)$ that
\begin{align}
\lfloor \tilde{a}(c) \rfloor \left(1-F_{J}(x\tilde{b}(c)^{-1}+\tilde{d}(c))\right) \Konvergenzcinfty -\log F_{A}(x).\label{LemmaKonvergenzLTCTRMGl2}
\end{align}
Now we can prove (a). With the change of variables $r=ct$ we get
\begin{align*}
&\intnu e^{-\xi t} P\left(\tilde{b}(c)(V(ct)-\tilde{d}(c)) \leq x \right) dt\\
&\hspace{5mm}=\frac{1}{c} \intnu e^{- (\xi c^{-1})r}P \left(V(r)\leq x \tilde{b}(c)^{-1}+\tilde{d}(c) \right) dr.
\end{align*}
Using Lemma \ref{lemmaLaplaceCTRMundOCTRM} (a) we then obtain for all $x_0<x<x_F$
\begin{align*}
&\intnu e^{-\xi t} P\left(\tilde{b}(c)(V(ct)-\tilde{d}(c)) \leq x \right) dt=\frac{1}{c}\frac{c}{\xi }\frac{1-\tilde{P}_W(\xi c^{-1})}{1-\mathcal{L}\left(P_{(W,J)}\right)(\xi c^{-1}, x\tilde{b}(c)^{-1}+\tilde{d}(c))}.
\end{align*}
By \eqref{LemmaKonvergenzLTCTRM2} and \eqref{LemmaKonvergenzLTCTRMGl1} we then get
\begin{align*}
\intnu e^{-\xi t} P\left(\tilde{b}(c)(V(ct)-\tilde{d}(c)) \leq x \right) dt&=\frac{1}{\xi} \frac{ \lfloor \tilde{a}(c) \rfloor  (1-\tilde{P}_W(\xi c^{-1}))}{\lfloor \tilde{a}(c) \rfloor(1-\mathcal{L}(P_{(W,J)})(\xi c^{-1},x\tilde{b}(c)^{-1}+\tilde{d}(c)))}\\
&\xrightarrow[c \rightarrow \infty]{} \frac{1}{\xi}\frac{\Psi_D(\xi)}{\Psi(\xi,x)}.
\end{align*}
Now we'll prove (b). As in the proof for (a) we obtain with  $r=ct$
\begin{align*}
\intnu e^{-\xi t} P\left(\tilde{b}(c)(U(ct)-\tilde{d}(c)) \leq x \right) dt =\frac{1}{c}\intnu e^{-(\xi c^{-1})r}P\left(U(r)\leq x{\tilde{b}(c)}^{-1}+\tilde{d}(c)\right) dr.
\end{align*}
With Lemma \ref{lemmaLaplaceCTRMundOCTRM} (b) we have for all $x_0<x<x_F$ 
\begin{align*}
&\intnu e^{-\xi t} P\left(\tilde{b}(c)(U(ct)-\tilde{d}(c)) \leq x \right) dt\\
&=\frac{1}{c}\frac{c}{\xi} \frac{F_J(x\tilde{b}(c)^{-1}+\tilde{d}(c))-\mathcal{L}(P_{(W,J)})(\xi c^{-1},x \tilde{b}(c)^{-1}+\tilde{d}(c))}{1-\mathcal{L}(P_{(W,J)})(\xi c^{-1},x \tilde{b}(c)^{-1}+\tilde{d}(c))}.
\end{align*}
Finally, using \eqref{LemmaKonvergenzLTCTRM2} and \eqref{LemmaKonvergenzLTCTRMGl2} we then have
\begin{align*}
&\intnu e^{-\xi t} P\left(\tilde{b}(c)(U(ct)-\tilde{d}(c)) \leq x \right) dt\\
&\hspace{5mm}=\frac{1}{\xi}\frac{\lfloor \tilde{a}(c) \rfloor (F_J(x\tilde{b}(c)^{-1}+\tilde{d}(c))-1)+\lfloor \tilde{a}(c) \rfloor \big(1-\mathcal{L}(P_{(W,J)})(\xi c^{-1},x\tilde{b}(c)^{-1}+\tilde{d}(c))\big)}{\lfloor \tilde{a}(c) \rfloor \big(1-\mathcal{L}(P_{(W,J)})(\xi c^{-1},x\tilde{b}(c)^{-1}+\tilde{d}(c))\big)}\\
&\hspace{5mm}\xrightarrow[c \rightarrow \infty]{}\frac{1}{\xi} \frac{\Psi(\xi,x)+\log F_{A}(x)}{\Psi(\xi,x)}.
\end{align*}
and the proof is complete.
\end{proof}
\end{lemma}
We are now in the position to prove Theorem \ref{Grenzwertverteilungsfunktion}.
\begin{proof}[Proof of Theorem \ref{Grenzwertverteilungsfunktion}] In view of Theorem \ref{Grenzwertsatz} we know that
\begin{align}
\left\{\tilde{b}(c)(V(ct)-\tilde{d}(c)) \right\}_{t>0} \JKonvergenzc \left\{A(E(t)-)^{+}\right\}_{t>0}.
\end{align} 
Since the $J_1$-convergence implies the convergence in all points in which the process is continuous in probability it follows
\begin{align*}
P_{\tilde{b}(c)(V(ct)-\tilde{d}(c))} \xrightarrow[c \rightarrow \infty]{w} P_{A(E(t)-)^{+}}
\end{align*}
in all but countable many $t>0$, due to the fact that the process has c\`adl\`ag paths and hence no more than countable many discontinuities. As a consequence
\begin{align*}
F_{\tilde{b}(c)(V(ct)-\tilde{d}(c))}(x) \xrightarrow[c \rightarrow \infty]{w} F_{A(E(t)-)^{+}}(x)
\end{align*}
pointwise in all but countable many $t>0$ and $x \in \mathbb{R}$. Therefore
\begin{align*}
\int_{0}^{\infty} e^{-\xi t} P\left(\tilde{b}(c)(V(ct)-\tilde{d}(c))\leq x \right) dt \xrightarrow[c \rightarrow \infty]{} \int_{0}^{\infty} e^{-\xi t} P \left(A(E(t)-)^{+} \leq x \right) dt 
\end{align*} 
in all but countable many $x \in \mathbb{R}$. By Lemma \ref{KonvergenzCLTransformierte} we know that
\begin{align*}
\int_{0}^{\infty} e^{-\xi t} P\left(\tilde{b}(c)(V(ct)-\tilde{d}(c)) \leq x \right) dt \xrightarrow[c \rightarrow \infty]{} \frac{1}{\xi}\frac{\Psi_D(\xi)}{\Psi(\xi,x)}
\end{align*}
for all $\xi>0$ and all $x_0<x<x_F$. Together with Lemma \ref{LemmazuGtundFt} it follows that
\begin{align*}
\int_{0}^{\infty} e^{-\xi t} P \left(A(E(t)-)^{+} \leq x \right) dt  = \int_{0}^{\infty} e^{-\xi t} G_t(x) dt
\end{align*}
for all but countable many $x \in \mathbb{R}$. Due to the uniqueness of the Laplace transform it follows
\begin{align} 
P\left(A(E(t)-)^{+} \leq x \right)=G_t(x) \label{GleichheitGrenzwert}
\end{align}
for all but countable many $t>0$ and $x\in \mathbb{R}$. Since the sample paths of the process $\left\{A(E(t)-)^{+}\right\}_{t>0}$ are c\`adl\`ag and hence right-continuous and due to the fact that $t\mapsto G_t(x)$ is right-continuous it follows that the equality in \eqref{GleichheitGrenzwert} holds for all $t>0$ and for all but countable many $x\in \mathbb{R}$. Since $P\left(A(E(t)-)^{+} \leq x \right)$ and $G_t(x)$ are distribution functions, they are right-continuous as functions in x. Hence it follows that the equality in \eqref{GleichheitGrenzwert} holds for all $x \in \mathbb{R}$ and  all $t>0$. The proof for the OCTRM is similar.
\end{proof}

The following corollary answers the question under which conditions the CTRM and OCTRM limit processes are equal.

\begin{cor}\label{idepcase}
Under the assumptions of Theorem \ref{Grenzwertverteilungsfunktion} the distributions of the CTRM and OCTRM processes are equal at any point $t>0$ in time
if and only if $D$ and $A$ in \eqref{Annahme} are independent.
\end{cor}

\begin{proof}
In view of \eqref{LTderVF} and \eqref{LTderVFOCTRW} are equal the distributions of $A(E(t)-)^+$ and $A(E(t))$ are equal for all $t>0$ if and only if 
$\Psi(\xi,x)=\Psi_D(\xi)-\log F_A(x)$ for all $\xi>0$ and $x_0<x<x_F$. By Corollary 3.10 of \cite{HeesScheffler1} this is equivalent to the independence of $D$ and $A$.
\end{proof}

\begin{remark}
In \cite{Pancheva} the case that there can be a dependence between the waiting times and the subsequent jumps and that the waiting times have infinite mean is also considered. There it is assumed that there exists a function $m(x)$ with $m(x) \rightarrow 1$ as $x \rightarrow x_F$, such that
\begin{align} \label{asymptotischeUnab}
\overline{F}_W(y) - \overline{F}_{(W,J)}(y,x) \sim m(x) \overline{F}_W(y) \hspace{0.5cm}  \text{ as } y \rightarrow \infty,
\end{align}
where $\overline{F}:=1-F$ is the tail function. But this condition is equivalent to asymptotic independence. It is enough to consider the 1-Fr\'echet case. In the following we assume that there exists $a_n>0$ such that
\begin{align*}
a_n S(n) \WKonvergenz D,
\end{align*}
where $D$ is $\beta$-stable for some $0<\beta<1$.
Assume further that there exist $b_n> 0$ such that
\begin{align*}
b_nM(n) \WKonvergenz A_1,
\end{align*}
where $A_1$ is $1$-Fr\'echet. 
It follows from \eqref{asymptotischeUnab}, that for $y>0$, $x>0$
\begin{align*}
&n \cdot P(a_n W>y, b_n J>x)\\
&\hspace{5mm}= n\cdot P( W> a_n^{-1} y,J>b_n^{-1}x)\\
&\hspace{5mm}=n \left[ P(W> a_n^{-1}y,J>b_n^{-1}x)-P(W>a_n^{-1}y)\right] + n \cdot P(W>a_n^{-1}y)\\
&\hspace{5mm}\sim -m(b_n^{-1}x) \cdot n \cdot P(W>a_n^{-1} y)+ n \cdot P(W>a_n^{-1}y)\\
&\hspace{5mm} \xrightarrow[n \rightarrow \infty]{} 0,
\end{align*}
because $b_n^{-1}x \rightarrow \infty$, $a_n^{-1}y  \rightarrow \infty$ as $n \rightarrow \infty$ and $m(x)\rightarrow 1$ as $x \rightarrow \infty$ uniform on compact subsets. The uniform compact convergence of $m(x)$ is not mentioned in \eqref{asymptotischeUnab} in \cite{Pancheva}, but is used in the proof for the limit distribution. Furthermore it follows that
\begin{align*}
&n \cdot P(a_n W>y,b_n J\geq 0)\\
&\hspace{5mm}=n \cdot P(a_n W>y)\\
&\hspace{5mm} \xrightarrow[n \rightarrow \infty]{} \Phi_{D}((y,\infty)).
\end{align*}
In addition we obtain
\begin{align*}
&n \cdot P(a_n W \geq 0 ,b_n J >x)\\
&\hspace{5mm}=n \cdot P(b_n J>x)\\
&\hspace{5mm} \xrightarrow[n \rightarrow \infty]{} x^{-1}.
\end{align*}
Hence the L\'evy-exponent measure of the limit distribution 
is concentrated on the coordinate axes. In view of \eqref{DAind} this is equivalent to  $A_1$ and $D$ being independent.
\end{remark}

\section{Governing Equations and Examples}

In this section we derive the so called {\it governing equation} for the distribution functions of the CTRM and OCTRM scaling limits. Those equations are time fractional pseudo differential equations whose solutions are those CDFs. Moreover we show in two examples, by explicit computations, the usefulness of the results in Theorem \ref{Grenzwertverteilungsfunktion}. Even though equations \eqref{VerteilungsfunktionCTRM} and \eqref{VerteilungsfunktionOCTRM} provide explicit formulas, they depend
on the joint distribution of $(D(t),A(t))$ which is hardly ever known explicitly. However, by Theorem \ref{Grenzwertverteilungsfunktion} we have
 for all $\xi>0$ and $x \in \mathbb{R}$ 
\begin{align}
\intnu e^{-\xi t}G(t,x)dt = \frac{1}{\xi}\frac{\Psi_D(\xi)}{\Psi(\xi,x)}, \label{GovEq1}
\end{align} 
and
\begin{align}
\intnu e^{-\xi t} F(t,x) dt = \frac{1}{\xi} \frac{\Psi(\xi, x)+\log F_A(x)}{\Psi(\xi, x)}, \label{GovEq2}
\end{align}
where
\begin{align*}
&G(t,x):=G_t(x)=P\left(A(E(t)-)^{+} \leq x\right) \\
\end{align*} 
 and
\begin{align*}  
&F(t,x):=F_t(x)=P\left(A(E(t)) \leq x\right).
\end{align*}

In the following we denote with $L(f)(\xi)=\int_0^\infty e^{-\xi t}f(t)\, dt$ the usual Laplace transform of a bounded and measurable function $f$. 
Then \eqref{GovEq1} and \eqref{GovEq2} are equivalent to
\begin{align}
\Psi(\xi,x)L(G(\cdot,x))(\xi)=\frac{1}{\xi} \Psi_D(\xi) \label{GovEq3}
\end{align}
and
\begin{align}
\Psi(\xi,x)L(F(\cdot,x))(\xi)=\frac{1}{\xi}(\Psi(\xi,x)+\log F_A(x)).\label{GovEq4}
\end{align}
for all $\xi>0$ and $x_0<x<x_F$.

If we now apply in \eqref{GovEq3} and \eqref{GovEq4} on both sides the inverse Laplace transform and assume w.l.o.g. that D is a $\beta$-stable subordinator with Laplace transform $E(e^{-sD})=\exp(-s^{\beta})$, we obtain for the distribution function of the CTRM limit
\begin{align}
\Psi(\partial_{t},x) G(t,x)=\frac{t^{-\beta}}{\Gamma(1-\beta)} \label{GovEq5}
\end{align}
and in view of Lemma \ref{lemmaCTRM1} for the distribution function of the OCTRM limit
\begin{align}
\Psi(\partial_t,x)F(t,x)=\Phi((t,\infty),\left[-\infty,x\right]). \label{GovEq6}
\end{align}
Here $\Psi(\partial_t,x)$ is a pseudo-differential operator which is defined by
\begin{align}
L(\Psi(\partial_t,x)u(\cdot,x))(\xi)=\Psi(\xi,x)L(u(\cdot,x))(\xi).
\end{align}
for suitable functions $u(t,x)$. 
We call \eqref{GovEq5} and \eqref{GovEq6} the \textbf{governing equations} of the distribution functions of the CTRM and OCTRM scaling limits. We call the distribution functions of the scaling limits the \textbf{mild solution} of the governing equation, if they  fulfill \eqref{GovEq3} and \eqref{GovEq4}, respectively. 
The fractional derivative $\partial^{\alpha}f/dx^{\alpha}$ for $0<\alpha<1$ of  suitable functions $f:\Rplus \rightarrow \mathbb{R}$ is defined as the function whose Laplace transform equals  $\xi^{\alpha}L(f)(\xi)$ for $\xi>0$.

\begin{example}\ \\
We first consider the case when the jumps $J_i$ and the waiting times $W_i$ are independent and hence $\left\{D(t)\right\}_{t>0}$ and $\left\{A(t)\right\}_{t>0}$ are independent as well. In view of Corollary \ref{idepcase} we have $G(t,x)=F(t,x)$ for all $t,x$.
In this case we obtain a governing equation with fractional time derivative. In the case that $D$ and $A$ are independent we have
\begin{align*}
\Psi(\xi,x)=\Psi_D(\xi)-\log F_A(x).
\end{align*}
Let $D$ be  $\beta$-stable which Laplace transform is given by $E(e^{-\xi D})=\exp(-\xi^{\beta})$ for some $0<\beta<1$ and let $\left\{A(t)\right\}_{t>0}$ be a $F$-extremal process, where $F_A(x)=\exp(-x^{-\alpha})$ is a Fr\'echet distribution. Hence 
\begin{align*}
\Psi(\xi,x)=\xi^{\beta}+x^{-\alpha}.
\end{align*}
In this case we get for the CTRM in \eqref{GovEq1} 
\begin{align}
L(F(\cdot,x))(\xi)=\frac{\xi^{\beta-1}}{\xi^{\beta}+x^{-\alpha}}. \label{GovBSP11}
\end{align}
To get the solution, that is the distribution function $F(t,x)$ of the (O)CTRM scaling limit, we apply on the right hand side in \eqref{GovBSP11} the inverse Laplace transform. Observe that
\begin{align}\label{GovBSP12}
\begin{split}
\frac{\xi^{\beta-1}}{\xi^{\beta}+x^{-\alpha}}&=\int_{0}^{\infty}\xi^{\beta-1}  e^{-(\xi^{\beta}+x^{-\alpha})u}du\\
&=\int_{0}^{\infty}\xi^{\beta-1}e^{-\xi^{\beta}u} (F_A(x))^{u}du.
\end{split}
\end{align}
Let $g_{\beta}(t)$ denote the density of $D$. Since $L(g_{\beta}(t))(\xi)=e^{-\xi^{\beta}}$ we have (cf. p.3 in \cite{Meerschaert2002a}) 
\begin{align*}
e^{-\xi^{\beta}u}&=\int_{0}^{\infty} e^{-\xi u^{1/\beta}v} g_{\beta} (v) dv\\
&=\int_{0}^{\infty} e^{-\xi t} g_{\beta}(u^{-1/\beta}t) u^{-1/\beta} dt.
\end{align*}
But $d(e^{-\xi^{\beta}u})/d\xi=-\beta \xi^{\beta-1}ue^{-\xi^{\beta}u}$ and hence
\begin{align*}
\xi^{\beta -1}e^{-\xi^{\beta}u}&=-\frac{1}{\beta u} \frac{d}{d\xi} \left( \int_{0}^{\infty} e^{-\xi t} g_{\beta}(u^{-1/\beta}t)u^{-1/\beta}dt\right)\\
&=\frac{1}{\beta u} \int_{0}^{\infty}  t e^{-\xi t} g_{\beta}(u^{-1/\beta}t) u^{-1/\beta} dt.
\end{align*}
Together with \eqref{GovBSP12} we can therefore write \eqref{GovBSP11} as
\begin{align*}
L(F(\cdot,x))(\xi)&=\int_{0}^{\infty} \left( \frac{1}{\beta u} \int_{0}^{\infty} t e^{-\xi t} g_{\beta}(u^{-1/\beta})u^{-1/\beta}dt\right)(F_A(x))^{u} du\\
&=\int_{0}^{\infty} e^{-\xi t} \left(\int_{0}^{\infty} (F_A(x))^{u} g_{\beta}(u^{-1/\beta}t) \frac{t}{\beta}u^{-1/\beta-1}du\right)dt.
\end{align*}
 Inverting the Laplace transform yields
\begin{align}
F(t,x)=\int_{0}^{\infty} (F_A(x))^{u} g_{\beta}(u^{-1/\beta}t) \frac{t}{\beta} u^{-1/\beta-1}du.\label{VFuncoupled}
\end{align}
This distribution function coincides with the distribution function of the scaling limit of the uncoupled CTRM in \cite{Meerschaert2007a}. 
Rewrite \eqref{GovBSP11} as
\begin{align*}
(\xi^{\beta}+x^{-\alpha})L(F(\cdot,x))(\xi)=\xi^{\beta-1},
\end{align*}
to see that the  governing equation in this case is given by
\begin{align*}
\frac{\partial^{\beta}}{\partial t^{\beta}} F(t,x) + x^{-\alpha}F(t,x)=\frac{t^{-\beta}}{\Gamma(1-\beta)}.
\end{align*}
\end{example}

\begin{example}\label{CTRMBspZFrechet} \ \\
Let $W$ be $\beta$-stable for some $0<\beta<1$ such that $\mathbb E[e^{-sW}]=\exp(-s^\beta)$, $s>0$. Furthermore let $Z$ be $\gamma$-Fr\'echet for some $\gamma>0$ with $P(Z\leq x)=\exp(-x^{-\gamma})$, $x>0$. Assume that $W$ and $Z$ are independent and set $J=W^{1/\gamma}Z$. Let $(W_i,J_i)$ be iid copies of $(W,J)$. Then, in view of Example 4.3 in \cite{HeesScheffler1} \eqref{Annahme} holds with $a_n=n^{-1/\beta}$ and $b_n=n^{-1/\beta\gamma}$. Moreover 
\begin{align*}
\Psi(\xi,x)=(\xi+x^{-\gamma})^{\beta}
\end{align*}
for $\xi,x>0$.
For the distribution function of the CTRM scaling limit we have by \eqref{GovEq1} that
\begin{align}
L(G(\cdot,x))(\xi)= \frac{\xi^{\beta-1}}{(\xi +x^{-\gamma})^{\beta}}. \label{CTRMBSP2Gl1}
\end{align}
Using
\begin{itemize}
\item[(i)] $L(f \ast g)(\xi)=L(f)(\xi)L(g)(\xi)$,
\item[(ii)] $L^{-1}(\xi^{\beta-1})(t)=t^{-\beta}/\Gamma(1-\beta)$,
\item[(iii)] $L^{-1}\Big((\xi+x^{-\gamma})^{-\beta}\Big)(t)=t^{\beta-1}e^{-x^{-\gamma}t}/\Gamma(\beta)$,
\end{itemize}
we obtain for the distribution function of the CTRM scaling limit 
\begin{align*}\label{VFCTRMFrechetFall}
G(t,x)=P(A(E(t)-)^{+} \leq x )&=\int_{0}^{t} e^{-ux^{-\gamma}} \cdot \frac{u^{\beta -1}(t-u)^{-\beta}}{\Gamma(\beta)\Gamma(1-\beta)}du.
\end{align*}
This is the distribution function of a random variable $Y$, where $Y\overset{d}{=}B^{1/\gamma}Z$, with $B$ Beta  distributed on $[0,t]$ and with density
\begin{align*}
b(u)=\frac{1}{\Gamma(\beta)\Gamma(1-\beta)}u^{\beta-1}(t-u)^{-\beta}
\end{align*}
and $Z$ is $\gamma$-Fr\'echet distributed with
$P(Z\leq x)=e^{-x^{-\gamma}}$ and independent of $B$. Rewrite \eqref{CTRMBSP2Gl1} as
\begin{align}
(\xi+x^{-\gamma})^{\beta} L(G(\cdot,x))(\xi)=\xi^{\beta-1}.
\end{align}
Applying the inverse Laplace transform on both sides, using
\begin{itemize}
\item[(i)] $L(\partial_{t}^{\beta}f(t))(\xi)=\xi^{\beta}L(f)(\xi)$ \text{ and }
\item[(ii)] $L(e^{-at}f(t))(\xi)=L(f)(\xi+a)$
\end{itemize}
the governing equation reads
\begin{align}
e^{-tx^{-\gamma}} \partial_{t}^{\beta} \left[e^{tx^{-\gamma}}G(t,x)\right]=\frac{t^{-\beta}}{\Gamma(1-\beta)}.
\end{align}
For the distribution function of OCTRM scaling limit we have by \eqref{GovEq2} that
\begin{align}
L(F(\cdot,x))(\xi)=\frac{1}{\xi}\frac{(\xi+x^{-\gamma})^{\beta}-x^{-\beta \gamma}}{(\xi+x^{-\gamma})^{\beta}}.\label{LTOCTRMBSP2}
\end{align}
In view of Lemma \ref{lemmaCTRM1} we have
\begin{align*}
L^{-1}\Big(\xi^{-1}(\xi+x^{-\gamma})^{\beta}-x^{-\beta\gamma})\Big)(t)&=\Phi((t,\infty) \times [0,x]),
\end{align*}
where $\Phi$ is the L\'evy measure of $(D,A)$. From example 4.3 in \cite{HeesScheffler1} we know that
\begin{align*}
\Phi(dr,ds)=(r^{1/\gamma}\omega)(ds) \frac{\beta}{\Gamma(1-\beta)} r^{-\beta-1} dr.
\end{align*}
with and $\omega=P_{Z}$. Hence we obtain
\begin{align*}
\Phi((t,\infty) \times [0,x])&=\int_{t}^{\infty} \int_{0}^{x} (r^{1/\gamma} \omega)(ds) \frac{\beta}{\Gamma(1-\beta)} r^{-\beta-1} dr\\
&=\int_{t}^{\infty} F_{Z}(r^{-1/\gamma}x)  \frac{\beta}{\Gamma(1-\beta)} r^{-\beta-1}dr\\
&=\int_{t}^{\infty} \exp(-rx^{-\gamma})  \frac{\beta}{\Gamma(1-\beta)} r^{-\beta-1} dr.
\end{align*}
Using $L(f \ast g)(\xi)=L(f)(\xi)L(g)(\xi)$ again, we get the distribution function of the OCTRM scaling limit, namely
\begin{align}\label{VFOCTRMFrechetFall}
F(t,x)=\int_{0}^{t}\int_{u}^{\infty} \frac{(t-u)^{\beta-1}e^{-(t-u)x^{-\gamma}}}{\Gamma(\beta)\Gamma(1-\beta)}\exp(-rx^{-\gamma})\beta r^{-\beta-1}drdu.
\end{align}
Rewrite \eqref{LTOCTRMBSP2} as 
\begin{align*}
(\xi+x^{-\gamma})^{\beta} L(F(\cdot,x))(\xi)=\frac{1}{\xi}((\xi+x^{-\gamma})^{\beta}-x^{-\gamma \beta}),
\end{align*}
to obtain the governing equation
\begin{align}
e^{-tx^{-\gamma}} \partial_{t}^{\beta} \left[e^{tx^{-\gamma}}G(t,x)\right]=\int_{t}^{\infty} \exp(-rx^{-\gamma}) \frac{\beta}{\Gamma(1-\beta)} r^{-\beta-1}dr.
\end{align}
\end{example}

\end{document}